\documentclass{article}

\usepackage[]{geometry}
\usepackage{amsfonts}
\usepackage{amssymb}
\usepackage{amsmath}
\usepackage{color}
\usepackage{fancyhdr}
\usepackage{amsthm}
\usepackage[hidelinks]{hyperref}
\usepackage{verbatim}
\usepackage{float}
\usepackage[all]{xy}
\usepackage{graphicx}
\usepackage{caption}
\usepackage{subcaption}
\usepackage{mathtools}
\usepackage{youngtab}
\usepackage{tikz}
\usepackage{comment}

\newtheorem{theorem}{Theorem}

\newtheorem{lemma}[theorem]{Lemma}
\newtheorem{proposition}[theorem]{Proposition}

\newtheorem{corollary}[theorem]{Corollary}
\newtheorem{conjecture}[theorem]{Conjecture}
\theoremstyle{definition}

\theoremstyle{example}

\newcommand{\bb}{\mathbb}
\newcommand{\mc}{\mathcal}
\newcommand{\Z}{\bb{Z}}

\usepackage{titling}

\title{On Anti-Powers in Aperiodic Recurrent Words}

\author{Aaron Berger\thanks{bergera@mit.edu}\\ \small\textit{Department of Mathematics, MIT}\\\small\textit{182 Memorial Drive, Cambridge, MA 02139}\\~\\
	Colin Defant\thanks{cdefant@princeton.edu}\\ \small\textit{Department of Mathematics, Princeton University}\\\small\textit{Fine Hall, 304 Washington Road, Princeton, NJ 08544}}

\date{}

\begin{document}
	\maketitle

\begin{abstract}
Fici, Restivo, Silva, and Zamboni define a \emph{$k$-anti-power} to be a concatenation of $k$ consecutive words that are pairwise distinct and have the same length. They ask for the maximum $k$ such that every aperiodic recurrent word must contain a $k$-anti-power, and they prove that this maximum must be 3, 4, or 5. We resolve this question by demonstrating that the maximum is 5. We also conjecture that if $W$ is a reasonably nice aperiodic morphic word, then there is some constant $C = C(W)$ such that for all $i,k\geq 1$, $W$ contains a $k$-anti-power with blocks of length at most $Ck$ beginning at its $i^\text{th}$ position. We settle this conjecture for binary words that are generated by a uniform morphism, characterizing the small exceptional set of words for which such a constant cannot be found.
This generalizes recent results of the second author, Gaetz, and Narayanan that have been proven for the Thue-Morse word, which also show that such a linear bound is the best one can hope for in general. 
\end{abstract}

\section{Introduction}
The problems we are concerned with in this paper arise in the study of combinatorics on infinite words, or  \emph{anti-Ramsey theory on $\Z$}. The original conception of Ramsey theory focused on unavoidable structures in colored graphs and began with Ramsey's work in 1930. Its extension to colorings of the integers has produced many notable results including the theorems of Roth and van der Waerden. Fici, Restivo, Silva, and Zamboni \cite{fici2018anti} describe Ramsey theory as an \textit{old and important} area of combinatorics; from this the observant reader may deduce that the variant they study, anti-Ramsey theory, is conversely \textit{new and exciting}. The study of anti-Ramsey theory was initiated by Erd\H{o}s, Simonovits, and S\'os in 1975 (one may debate whether 1975 qualifies as ``new'' in combinatorics), and the recent work of Fici et al. has been the impetus for a flood of new activity in the area \cite{badkobeh2018algorithms,
burcroff2018k,defant2017anti,
fici2018abelian,gaetz2018anti,
kociumaka2018efficient,
narayanan2017functions}. Specifically, the notion that has attracted this activity is that of a \textit{$k$-anti-power}, which Fici et al. define to be a word formed by concatenating $k$ consecutive pairwise-distinct factors (i.e., a word of length $km$ that can be partitioned into pairwise-distinct contiguous ``blocks'' of size $m$). 

An infinite word $W$ is \emph{aperiodic} if it is not eventually periodic, and it is \textit{recurrent} if every finite factor of $W$ occurs infinitely often in $W$. We say $W$ is \emph{uniformly recurrent} if for every finite factor $w$ of $W$, there is a positive integer $n$ such that every factor of $W$ of length $n$ contains $w$ as a factor. In their foundational work, Fici et al. demonstrate, among other results, three fundamental properties of anti-powers in infinite words:
\begin{theorem}[Fici, Restivo, Silva, Zamboni \cite{fici2018anti}]

	~
	\begin{enumerate}
		\item (Corollary 11) Every infinite aperiodic word contains a 3-anti-power.
		\item (Proposition 12) There exist aperiodic infinite words avoiding 4-anti-powers.
		\item (Proposition 13) There exist infinite aperiodic recurrent words avoiding 6-anti-powers.
	\end{enumerate}
	
\end{theorem}
It has remained unknown whether every infinite aperiodic recurrent word must contain a 4-anti-power or a 5-anti-power; in this paper, we show the stronger statement of the two, thereby closing completely the gap between the lower bound and upper bound:
\begin{theorem}\label{main theorem}
	Every infinite aperiodic recurrent word contains a 5-anti-power.
\end{theorem} 

A natural question to investigate next is which restrictions on words guarantee longer anti-power factors.  One obvious direction to take concerns morphic words, which we define next. These form a well-studied collection of words that are often aperiodic and recurrent. Indeed, morphic words originally provided motivation for the study of general aperiodic recurrent words.  

Let $\mathcal A^*$ denote the set of all finite words over the alphabet $\mathcal A$ (i.e., the free monoid generated by $\mathcal A$). A \emph{morphism} is a map $\mu:\mathcal A^*\to\mathcal A^*$ with the property that $\mu(ww')=\mu(w)\mu(w')$ for all $w,w'\in\mathcal A^*$. A morphism is uniquely determined by specifying its values on the letters in $\mathcal A$. For example, if $\mathcal A=\{0,1\}$, then $\mu(0110)=\mu(0)\mu(1)\mu(1)\mu(0)$. Given $a\in\mathcal A$, a morphism $\mu$ is said to be \emph{prolongable at $a$} if $\mu(a)=as$ for some nonempty word $s$. If $\mu$ is prolongable at $a$, then the sequence $a,\mu(a),\mu^2(a),\ldots$ converges to the infinite word $\mu^\omega(a)$. An infinite word $W$ is called \emph{pure morphic} if $W=\mu^\omega(a)$ for some morphism $\mu$ that is prolongable at $a$. In this case, we also say $W$ is \emph{generated} by $\mu$. 

A morphism $\mu:\mathcal A^*\to\mathcal A^*$ is called \emph{$r$-uniform} if $\mu(a)$ has length $r$ for every $a\in\mathcal A$. A morphism is simply called \emph{uniform} if it is $r$-uniform for some $r$. An infinite word is called \emph{morphic} if it is the image under a $1$-uniform morphism (also called a \emph{coding}) of a pure morphic word. In Section \ref{Sec:Morphic}, we consider a binary word $W$ that is generated by a uniform morphism $\mu$. In order for this to make sense, $\mu$ must be $r$-uniform for some $r\geq 2$ (otherwise, it is not prolongable). We refer the reader to \cite{ alloucheshallit, bugeaud2011morphic} for more information about morphic words, uniform morphisms, and their connections with automatic sequences.

We state a conjecture here, left purposefully vague.
\begin{conjecture}\label{Conj1}
	If $W$ is a sufficiently well-behaved aperiodic morphic word, then there is a constant $C = C(W)$ such that for all positive integers $i$ and $k$, $W$ contains a $k$-anti-power with blocks of length at most $Ck$ beginning at its $i^\text{th}$ position.
\end{conjecture}

Corollary 7 of \cite{fici2018anti} gives a similar result without the uniform linear bound. The works of the second author and Narayanan  \cite{ defant2017anti, narayanan2017functions} confirm and extend this conjecture when $W={\bf t}$ is the famous Thue-Morse word and $i=1$. The subsequent results of Gaetz \cite{gaetz2018anti} confirm this conjecture for $W={\bf t}$ and for every fixed $i$ with a constant $C$ that could depend on $i$. More precisely, let $\gamma_{i-1}(k)$ denote the smallest positive integer $m$ such that the factor of ${\bf t}$ of length $km$ beginning at the $i^\text{th}$ position of ${\bf t}$ is a $k$-anti-power. Gaetz proved that \[\frac{1}{10}\leq\liminf_{k\to\infty}\frac{\gamma_{i-1}(k)}{k}\leq\frac{9}{10}\quad\text{and}\quad\frac{1}{5}\leq\limsup_{k\to\infty}\frac{\gamma_{i-1}(k)}{k}\leq\frac{3}{2}\] for every positive integer $i$. The lower bounds in these estimates show that the linear upper bound in Conjecture \ref{Conj1} is the best one can hope to prove in general. 

In Section \ref{Sec:Morphic}, we settle Conjecture \ref{Conj1} in the case in which $W$ is a binary word that is generated by a uniform morphism. More precisely, we will see that the conjecture holds in all but a few exceptional cases that are characterized in the following proposition. In this proposition, we assume our binary word begins with $0$, but the analogous statement certainly holds if the word starts with $1$ and we switch the roles of the letters $0$ and $1$ everywhere. 

\begin{proposition}\label{Prop1}
Let $W$ be a binary word that starts with $0$ and is generated by an $r$-uniform morphism $\mu$. Then:
\begin{itemize}
    \item  $W$ is aperiodic if and only if $\mu(0)\neq\mu(1)$ and $W\not\in\{0000\cdots, 0111\cdots, 0101\cdots\}$.
    \item $W$ is uniformly recurrent if and only if it is $0000\cdots$ or $\mu(1)\neq 11\cdots 1$.
\end{itemize}
\end{proposition}
Let us remark that Conjecture \ref{Conj1} is easily seen to fail in both of these exceptional cases. Observing the words $W$ that fail to be aperiodic, we see that each has at most $r$ distinct factors of any length, and so cannot have $k$-anti-powers for $k > r$. If $W$ fails to be uniformly recurrent, it follows from the above characterization that $W$ has constant factors of arbitrary length, inside which one certainly cannot find anti-powers of bounded length.

The following theorem verifies Conjecture \ref{Conj1} for all binary words that are generated by a uniform morphism and that do not lie in the set of exceptional words listed in Proposition \ref{Prop1}.  
\begin{theorem}\label{morphictheorem}
If $W$ is a uniformly recurrent aperiodic binary word that is generated by a uniform morphism, then there is a constant $C=C(W)$ such that for all positive integers $i$ and $k$, $W$ contains a $k$-anti-power with blocks of length at most $Ck$ beginning at its $i^\text{th}$ position.
\end{theorem}

\subsection*{Terminology}
Our words are always taken to be sequences of characters (letters) from a finite alphabet $\mc A$. We say a word is \emph{binary} if it is a word over the two-element alphabet $\{0,1\}$. By ``infinite," we always mean infinite to the right. To reiterate, a \emph{factor} of a word is a contiguous subword. Throughout this article, we let $[i,j]$ denote the factor of the infinite word $W$ that starts in the $i^\text{th}$ position of $W$ and ends in the $j^\text{th}$ position. A \emph{prefix} of a word is a factor that contains the first character, and a \emph{suffix} of a finite word is a factor that contains the final character. We let $|w|$ denote the length of a finite word $w$. 

\section{Constructing 5-Anti-Powers}

We will prove Theorem \ref{main theorem} by constructing a 5-anti-power in an arbitrary aperiodic recurrent word $W$. For this construction, it will be essential to find some ``anchor points'' in $W$ that will allow us to get our bearings, so to speak. For example, in the periodic word $010101\cdots$, the factors $[i,j]$ and $[i+2,j+2]$ are always identical. It will be useful for us to prohibit this from happening:
\begin{lemma}\label{spaced out}
	For every infinite aperiodic word $W$ and every $t > 0$, there is a word $w$ with the following properties:
	\begin{itemize}
		\item A copy of $w$ appears as a factor of $W$.
		\item If a factor of $W$ beginning at index $i$ is equal to $w$, then no factor of $W$ beginning at any of the indices $i+1,\ldots,i+t$ is equal to $w$.
	\end{itemize}
\end{lemma}

In order to prove this lemma, we appeal to a stronger statement of Ehrenfeucht and Silberger.\footnote{A straightforward induction would also suffice.} Say a word is \textit{unbordered} if no nontrivial prefix (i.e., no prefix other than the empty word and the full word) is also a suffix. 
	\begin{theorem}[\cite{ehrenfeucht1979periodicity}, Theorem 3.5]\label{ehrenfeuchttheorem}
		If $W$ is an infinite word and $t \in \Z$ is such that all unbordered factors of $W$ have length at most $t$, then $W$ is eventually periodic.
	\end{theorem}

\begin{proof}[Proof of Lemma \ref{spaced out}]
		As an immediate corollary of Theorem \ref{ehrenfeuchttheorem}, we obtain that any infinite aperiodic word $W$ has an unbordered factor $w$ of length $\ell > t$. If an occurrence of $w$ began at index $i$ and another began at index $j \in \{i+1,\ldots,i+t\}$, then the factor $[j, i+\ell - 1]$ would be a nontrivial prefix of the second occurrence of $w$ and a nontrivial suffix of the first occurrence, contradicting the fact that $w$ is unbordered. We conclude that such a choice of $w$ satisfies the conditions of the lemma.
\end{proof}
We proceed to the proof that aperiodic recurrent words contain $5$-anti-powers.
\begin{proof}[Proof of Theorem \ref{main theorem}]
	As before, let $W$ be an aperiodic recurrent word and $w$ a factor guaranteed by Lemma \ref{spaced out} for $t = 100$. Let $\ell=|w|$.
	
	Since $W$ is recurrent, we can find a pair of occurrences of $w$ that are a distance $d_1 \ge \ell+1000$ apart. Again by recurrence, we can find a second copy of this pair that is at a distance $d_2 \ge 10d_1$ away from the first occurrence of this pair. Finally, applying the fact that $W$ is recurrent once again, we can find a copy of these four occurrences of $w$ that begins at an index $i_1 \ge d_2$. We have now identified four indices $i_1< i_2< i_3< i_4$ at which occurrences of $w$ begin such that $$i_2-i_1=i_4-i_3 =: d_1 \ge \ell+1000, \quad i_3-i_2 =:d_2 \ge 10d_1,\text{~~and~~} i_1 \ge d_2.$$ Figure \ref{sketch} gives a sketch of the anti-power we will construct.
	
	\begin{figure}[H]
		\begin{center}
			\includegraphics[width=.9\textwidth]{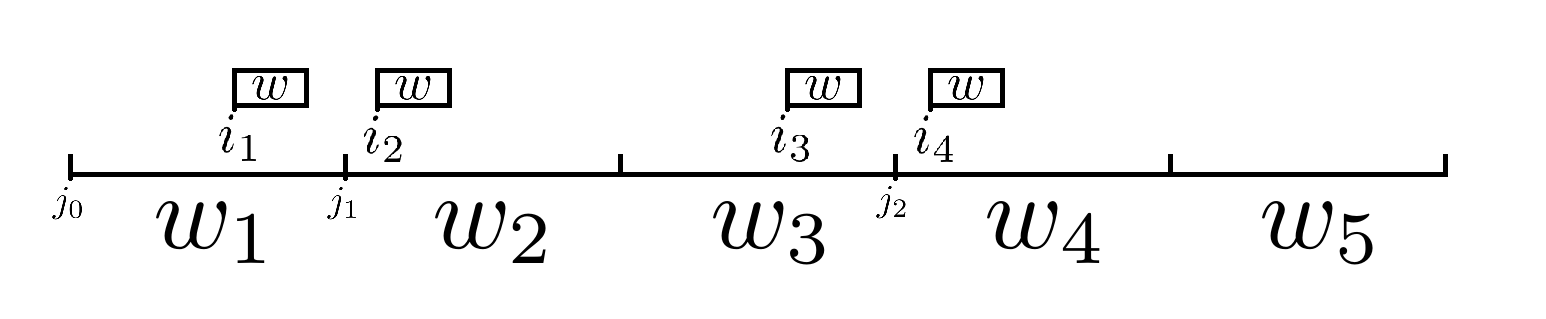}
		\end{center}
	\caption{The 5-anti-power we construct.}\label{sketch}
	\end{figure}

	Now let $j_1 := i_1+\ell+500$ and $j_2 \in \{i_3+\ell+500, i_3+\ell+501\}$ be such that $j_2 - j_1$ is even. Let $D = (j_2-j_1)/2$, and set $j_0 = j_1 - D$. (It readily follows from our construction that $j_0$ is positive.) We now construct 11 \textit{potential} anti-powers starting at $j_0$, each comprising 5 consecutive factors, called \textit{blocks}, of $W$. We then show that for at least one of these 11, all $\binom{5}{2}$ pairs of blocks are distinct. For $i \in \{1,\ldots,5\}$ and $c \in \{0,\ldots,10\}$, the $i^\text{th}$ block of the $c^\text{th}$ construction is $$w_i^{(c)} := [j_0+(i-1)(D+c),j_0+i(D+c) - 1].$$
	As a motivating example, setting $c = 0$, we see the $0^\text{th}$ construction is given by the following 5 blocks:
	\begin{align*}
	w_1^{(0)} = [j_1-D, j_1 - 1],\qquad w_2^{(0)} = [j_1,& j_1 +D - 1],\qquad w_3^{(0)} = [j_2-D, j_2 - 1],\\w_4^{(0)} = [j_2, j_2+D-1],\quad\quad&\quad\quad w_5^{(0)} = [j_2+D, j_2 +2D-1].
	\end{align*}

	A property of this construction, and its main purpose, is that $w_1^{(0)},w_2^{(0)},w_3^{(0)},w_4^{(0)}$ all contain copies of $w$ such that (every letter in) each copy of $w$ is more than 100 spaces away from either endpoint of the block that contains it. This is an immediate consequence of our choices of $j_1$ and $j_2$ to be between specific occurrences of $w$. For example, $w_2^{(0)}$ begins at least 500 indices before $i_2$, and at most $d_1$ indices before $i_2$. Since it is of length $D \ge d_2/2 > 2d_1 \ge d_1+\ell+1000$, it also ends at least 1000 indices after the copy of $w$ beginning at $i_2$ ends. The other 3 cases proceed similarly.
	
	Now let's see what happens for $w_a^{(c)}$ for other $c$. The maximum amount that any endpoint changes when compared to $w_a^{(0)}$ is $50$, which is the distance moved by the right endpoint of $w_5^{(10)}$. Thus, for every $a \in \{1,\ldots,4\}$ and $c \in \{1,\ldots,10\}$, the block $w_a^{(c)}$ also fully contains the same copy of $w$ identified in the corresponding block $w_a^{(0)}$, and this copy is at a distance of at least 50 from either endpoint.
	
	We proceed to show that one of these constructions produces an anti-power. For each $c$ where an anti-power is not produced, we must have $w_a^{(c)} = w_b^{(c)}$ for some $a , b \in \{1,\ldots,5\}$ with $a<b$. Then since $a < 5$, $w_a^{(c)}$ contains one of the copies of $w$ identified above, beginning some number $i$ of indices after the beginning of the block. Then $w_b^{(c)}$ must contain a copy of $w$ beginning at its $i^\text{th}$ letter as well. We claim this implies $w_a^{(c')} \neq w_b^{(c')}$ for all $c' \neq c$ in $\{0,\ldots,10\}$. Since the endpoints of these blocks again change by less than 50, both $w_a^{(c')}$ and  $w_b^{(c')}$ still fully contain the copies of $w$ that we identified in $w_a^{(c)}$ and $w_b^{(c)}$, although now these copies of $w$ begin at new relative indices: $i + (a-1)(c' - c)$ and $i + (b-1)(c' - c)$, respectively. If we assume for the sake of contradiction that $w_a^{(c')} = w_b^{(c')}$, then both blocks must contain an appearance of $w$ at both of these indices. However, these indices differ by $(a-b)(c'-c) \le (5-1)(10) < 100$, which contradicts our assertion that consecutive appearances of $w$ must appear at distance greater than 100.
	
	Consequently, each pair $a < b$ can satisfy $w_a^{(c)} = w_b^{(c)}$ for at most one value of $c$. There are $\binom {5}{2} = 10$ pairs of $a < b$ and 11 choices for $c$. This means that at least one choice of $c$ must have no such pairs, and therefore must result in an anti-power.
\end{proof}

\textit{Remark.} It may be of interest to discuss why this proof cannot be extended to construct a 6-anti-power. One source of intuition for this fact is as follows. When constructing our anti-power, we needed to force 4 of the 5 blocks to contain a copy of a specifically chosen $w$. We have two degrees of freedom when constructing an anti-power; this allows us to carefully place two of the endpoints out of the set of endpoints of the blocks of the anti-power. Since each endpoint is adjacent to a pair of blocks, this gives us fine control over at most 4 blocks. It turns out this is the best one can do; Fici et. al. \cite{fici2018anti} construct a word so that among any six consecutive blocks of equal length, there are two that are not only identical, but constant. 

\section{Anti-Powers in Binary Morphic Words}\label{Sec:Morphic}

We begin this section by establishing some additional notation. Let us fix an infinite aperiodic binary word $W$ that is generated by a uniform morphism $\mu:\{0,1\}^*\to\{0,1\}^*$. As mentioned at the end of the introduction, $\mu$ is $r$-uniform for some $r\geq 2$. We can write \[\mu(0)=A=A_1\cdots A_r,\quad\mu(1)=B=B_1\cdots B_r,\] where $A_1,\ldots,A_r,B_1,\ldots,B_r\in\{0,1\}$. We may assume that the first letter of $W$ is $0$ and that $A_1=0$ (i.e., $\mu$ is prolongable at $0$). Thus, $W=\mu^\omega(0)$. We must have $A\neq B$. Indeed, otherwise, we would have $W=AAAA\cdots$, contradicting the assumption that $W$ is aperiodic. As before, we write $[ i,j]$ to refer to the factor of $W$ beginning at index $i$ and ending at index $j$. One important point to keep in mind is that for each nonnegative integer $t$, the factor $[tr+1,tr+r]$ is equal to either $A$ or $B$ because it is the image under $\mu$ of the $(t+1)^\text{st}$ letter in $W$. 

We now proceed to prove Theorem \ref{morphictheorem}, which states that Conjecture \ref{Conj1} holds if $W$ is uniformly recurrent. We start with some lemmas. 
\begin{lemma}\label{Lem0}
	Let $W$ be an aperiodic binary word generated by an $r$-uniform morphism $\mu$ with $\mu(0)=A$ and $\mu(1)=B$. If $[\gamma+1,\gamma+3r]$ is equal to $AAB$ or $BBA$, then $r$ divides $\gamma$.
\end{lemma} 
\begin{proof}
	We only consider the case in which $[\gamma+1,\gamma+3r]=AAB$; the proof is similar when $[\gamma+1,\gamma+3r]=BBA$. Suppose instead that $r$ does not divide $\gamma$, and let $h \in \{1,\ldots, r - 1\}$ be such that $r$ divides $\gamma + h$. Let $D$, $E$, and $F$ be the three consecutive factors of length $r$ starting at index $\gamma+h+1$. That is, $D = [\gamma+h+1, \gamma+h+r]$, $E = [\gamma+h+r+1, \gamma+h+2r]$, and $F = [\gamma+h+2r+1, \gamma+h+3r]$. Then $D$, $E$, and $F$ are each images of a single letter under $\mu$, so each is equal to either $A$ or $B$. 
	
	Because $W$ is aperiodic, we know that $A\neq B$. We are going to prove by induction on $j$ that 
	\begin{equation}\label{Eq2}
	A_1\cdots A_j=B_1\cdots B_j
	\end{equation} 
	for all $j\in\{1,\ldots,r\}$, which will yield our desired contradiction.
	
	Assume for the moment that $D=E$. Comparing the overlaps between $A$ and $D$ and between $B$ and $E$, we find that \[A_1\cdots A_h=[\gamma+r+1,\gamma+r+h]=D_{r-h+1}\cdots D_r=E_{r-h+1}\cdots E_r=[\gamma+2r+1,\gamma+2r+h]=B_1\cdots B_h.\] This proves \eqref{Eq2} for all $j\in\{1,\ldots,h\}$, completing the base case of our induction. Now choose $n\in\{h,\ldots,r-1\}$, and assume inductively that we have proven \eqref{Eq2} when $j=n$. We will prove \eqref{Eq2} when $j=n+1$, which will complete the inductive step. Of course, this amounts to proving that $A_{n+1}=B_{n+1}$, since we already know by induction that $A_1\cdots A_n=B_1\cdots B_n$. We determine the indices of the overlaps of $A$ and $B$ with $D$ and $F$, computing $A_{n+1} = D_{n - h + 1}$ and $B_{n+1} = F_{n - h + 1}$. Since $D,F \in \{A,B\}$, we have $D_{n-h+1},F_{n-h+1}\in\{A_{n-h+1},B_{n-h+1}\}$. Our induction hypothesis now tells us that $A_{n-h+1}=B_{n-h+1}$. It follows that $D_{n-h+1}=F_{n-h+1}$, which completes this case of the proof. 
	
	We now consider the case in which $D\neq E$. This implies that $\{D,E\}=\{A,B\}$. Comparing the overlaps of both copies of $A$ with $D$ and $E$, we see \[D_1\cdots D_{r-h}=[\gamma+h+1,\gamma+r]=A_{h+1}\cdots A_r=[\gamma+r+h+1,\gamma+2r]=E_1\cdots E_{r-h}.\]
	Since $\{D,E\}=\{A,B\}$, this proves that $A_1\cdots A_{r-h}=B_1\cdots B_{r-h}$. This proves \eqref{Eq2} for all $j\in\{1,\ldots,r-h\}$, completing the base case of our induction. Now choose $n\in\{r-h,\ldots,r-1\}$, and assume inductively that we have proven \eqref{Eq2} when $j=n$. We will prove \eqref{Eq2} when $j=n+1$, which will complete the inductive step. Of course, this amounts to proving that $A_{n+1}=B_{n+1}$ since we already know by induction that $A_1\cdots A_n=B_1\cdots B_n$. Computing overlaps once more, we find $D_{n+1}=A_{n+h+1-r}$ and $E_{n+1}=B_{n+h+1-r}$. Our induction hypothesis tells us that $A_{n+h+1-r}=B_{n+h+1-r}$, so $D_{n+1}=E_{n+1}$. Since $\{D,E\}=\{A,B\}$, this implies that $A_{n+1}=B_{n+1}$ as desired. 
\end{proof}

\begin{lemma}\label{Lem1}
Let $W$ be an aperiodic binary word generated by an $r$-uniform morphism. There is an integer $c_1=c_1(W)\geq 1$ with the following property. If $X$ is a word that begins at indices $i_1$ and $i_2$ in $W$ and $|X|\geq rc_1+2r - 2$, then $r$ divides $i_2-i_1$.  
\end{lemma} 

\begin{proof}
As before, let $\mu$ be an $r$-uniform morphism that generates $W$, and let $\mu(0)=A$ and $\mu(1)=B$. Because $W$ is aperiodic, it is easy to verify that $W$ contains either $001$ or $110$ as a factor. Let us assume $W$ contains $001$; the proof is similar if we assume instead that it contains $110$. Because $W$ is uniformly recurrent, there exists $c_1\geq 1$ such that every factor of $W$ of length at least $c_1$ contains $001$. Now let $X$ be a factor of $W$ with $|X| = m \ge rc_1+2r - 2$, and assume $X=[i_1,i_1+m-1]=[i_2,i_2+m-1]$. It will again be helpful to look at factors of $W$ of the form $[tr+1, tr+r]$ since each such factor is equal to either $A$ or $B$.

Let $n$ be the unique multiple of $r$ in $\{i_1-1, \ldots, i_1+r -2\}$ and $n'$ be the unique multiple of $r$ in $\{i_1+m-r, \ldots, i_1+m-1\}$. Note that $n' - n \ge m-2r+2 \ge rc_1$ and that $[n+1,n']$ is a factor of $X$. Since $r$ divides $n$ and $n'$, $[n+1,n']$ is the image of a word $u$ under the map $\mu$. Moreover, $|u| =(n'-n)/r\ge c_1$. Our choice of $c_1$ guarantees that $u$ contains 001 as a factor, so $[n+1,n']$ must contain $AAB$ as a factor. Thus, $X$ contains $AAB$ as a factor. Let us say a copy of $AAB$ starts at the $\ell^\text{th}$ letter of $X$. Since $X=[i_1,i_1+m-1]=[i_2,i_2+m-1]$, we have $AAB=[i_1+\ell-1,i_1+\ell+3r-2]=[i_2+\ell-1,i_2+\ell+3r-2]$. Lemma \ref{Lem0} now guarantees that $r$ divides both $i_1+\ell-2$ and $i_2+\ell-2$, which implies that $r$ divides $i_2-i_1$.  
\end{proof}

\begin{corollary}\label{Cor1}
Let $\alpha$ be a positive integer. Let $W$ be an aperiodic binary word generated by an $r$-uniform morphism $\mu$, and let $c_1=c_1(W)$ be the constant from Lemma \ref{Lem1}. If $X$ is a word that begins at indices $i_1$ and $i_2$ in $W$ and $|X|\geq r^\alpha c_1+2r^\alpha - 2$, then $r^\alpha$ divides $i_2-i_1$.  
\end{corollary}

\begin{proof}
	Note that if $W$ is generated by $\mu$, then it is also generated by $\mu^\alpha$. Since $\mu^\alpha$ is $r^\alpha$-uniform, the desired result follows immediately from Lemma \ref{Lem1} with $\mu$ replaced by $\mu^\alpha$ and $r$ replaced by $r^\alpha$. 
\end{proof}

We are now in a position to prove Theorem \ref{morphictheorem}. 

\begin{proof}[Proof of Theorem \ref{morphictheorem}]
As before, we let $W$ be an infinite aperiodic uniformly recurrent binary word that is generated by a morphism that is $r$-uniform for some $r\geq 2$. Let $c_1=c_1(W)$ be the constant from Lemma \ref{Lem1}, and let $C=(c_1+2)r$. Fix positive integers $i$ and $k$. Our goal is to show that $W$ contains a $k$-anti-power with blocks of length at most $Ck$ beginning at its $i^\text{th}$ position. This is obvious if $k=1$, so we may assume $k\geq 2$.

Let $\alpha$ be the unique positive integer such that $r^{\alpha-1}<k\leq r^\alpha$. Let $U=[i,k((c_1+2)r^\alpha-1)+i-1]$ be the factor of $W$ of length $k((c_1+2)r^\alpha-1)$ that begins at the $i^\text{th}$ position of $W$. We can write $U=U^{(1)}\cdots U^{(k)}$, where $|U^{(j)}|=(c_1+2)r^\alpha-1$ for all $j\in\{1,\ldots,k\}$. Suppose $U^{(j)}=U^{(j')}$ for some $j,j'\in\{1,\ldots,k\}$. Because $|U^{(j)}|=(c_1+2)r^\alpha-1\geq r^\alpha c_1+2r-2$, Corollary \ref{Cor1} tells us that $r^\alpha$ divides the difference between the index where $U^{(j)}$ starts and the index where $U^{(j')}$ starts. This difference is $(j'-j)((c_1+2)r^\alpha-1)$. At this point, we invoke the crucial fact that $(c_1+2)r^\alpha-1$ and $r^\alpha$ are relatively prime ($c_1$ is an integer). This means that $r^\alpha$ divides $j'-j$. Since $j,j'\in\{1,\ldots,k\}\subseteq\{1,\ldots,r^\alpha\}$, we must have $j=j'$. It follows that the blocks $U^{(1)},\ldots,U^{(k)}$ are pairwise distinct. We conclude the proof by observing that these blocks are of length $(c_1+2)r^\alpha-1<(c_1+2)rk = Ck$, as desired.
\end{proof}

\subsection*{Exceptional Words}
We conclude with the characterization of exceptional cases discussed in the introduction.
\begin{proof}[Proof of Proposition \ref{Prop1}]
Let $W$ be a binary word that starts with $0$ and is generated by an $r$-uniform morphism $\mu$. Suppose first that $W$ is uniformly recurrent and not equal to $0000\cdots$. Since $W$ is uniformly recurrent and contains $0$, it cannot contain arbitrarily long factors consisting of only 1's. However, it is easy to verify that such factors occur if $\mu(1)=11\cdots 1$. Hence, $\mu(1)\neq 11\cdots 1$. 

To prove the converse, let us assume that $W$ is not uniformly recurrent. This means that $W$ contains a factor $w$ such that $W$ contains arbitrarily long factors that do not contain $w$. There is a positive integer $\alpha$ such that $w$ is a factor of the prefix of $W$ of length $r^\alpha$. This prefix is $\mu^\alpha(0)$. Because there are arbitrarily long factors of $W$ that do not contain $w$, there are arbitrarily long factors that do not contain $\mu^\alpha(0)$. Because $W$ is generated by $\mu$, there are arbitrarily long factors of $W$ that do not contain $0$. This clearly cannot be the case if $\mu(1)$ contains 0 (since $\mu(0)$ necessarily contains 0), so we must have $\mu(1) = 11\cdots1$. Of course, this also implies that $W\neq 0000\cdots$.  

It remains to verify the characterization of aperiodic words. The words $0000\cdots$, $0111\cdots$, and $0101\cdots$ are not aperioidic (i.e., they are eventually periodic). Furthermore, if $\mu(0)=\mu(1)$, then $W=\mu(0)\mu(0)\mu(0)\cdots$ is periodic. This proves one direction. 

For the converse, assume $\mu(0)\neq \mu(1)$ and $W\not\in\{0000\cdots, 0111\cdots, 0101\cdots\}$. We first assume $W$ is uniformly recurrent. It is easy to verify that $W$ must contain either 001 or 110. In the proof of Theorem \ref{morphictheorem} (and the results immediately preceding it), the only times we used the aperiodicity of the word under consideration were when we wanted to deduce that $A\neq B$ and that the word contains either $001$ or $110$. This means that the same proof applies to our word $W$ to show that $W$ contains arbitrarily long anti-powers starting at every index. If $W$ were periodic with a period of $k$ past some index $i$, then it would only have at most $k$ distinct factors of any fixed size beginning at index $i$. This would exclude the appearance of $(k+1)$-anti-powers starting at $i$, which would be a contradiction. Hence, $W$ is aperiodic. 

Now assume $W$ is not uniformly recurrent. It follows from the first part of this proof that $\mu(1)=11\cdots 1$ and that $W$ contains arbitrarily long factors that do not contain $0$. On the other hand, since $W\neq 0111\cdots$ and $\mu(0)$ contains 0, we can prove inductively that $W$ contains infinitely many zeros. No eventually periodic word can satisfy both of these conditions simultaneously, so the proof is complete.
\end{proof}

\section{Further Work}
In Section \ref{Sec:Morphic}, we settled part of our vague Conjecture \ref{Conj1}. An obvious next step would involve proving additional cases of this conjecture. One way to do this would be to remove the ``binary" condition from Theorem \ref{morphictheorem}. That theorem also specifies that the word under consideration is generated by a uniform morphism; one could attempt to remove this ``uniform" condition. Of course, we would certainly like to see the conjecture proved in its entirety. Since the conjecture is not completely precise, this would amount to classifying those morphic words $W$ for which there exists a constant $C(W)$ satisfying the property stated in the conjecture.   
	
	\section{Acknowledgments}
	The authors would like to thank Amanda Burcroff for providing many helpful comments on the first drafts of this paper. The second author was supported by a Fannie and John Hertz Foundation Fellowship and an NSF Graduate Research Fellowship.

\bibliographystyle{acm}
\bibliography{recurrent_antipowers_bib}

\end{document}